\documentclass[leqno,12pt]{amsart}
\usepackage[letterpaper,margin=1in,footskip=0.25in]{geometry}
\usepackage[utf8]{inputenc}
\usepackage{Baskervaldx}
\usepackage[cal=boondoxupr]{mathalfa}
\usepackage{mathrsfs}
\usepackage{enumerate, setspace, amssymb, amsopn, amsmath, array, pdfsync, url, amsfonts, float}
\usepackage[baskervaldx,bigdelims,vvarbb]{newtxmath}

\usepackage[colorlinks=true, citecolor=RoyalBlue, filecolor=RoyalBlue, linkcolor=RoyalBlue, urlcolor=RoyalBlue, hyperindex]{hyperref}
\usepackage[dvipsnames]{xcolor}
\usepackage{tikz, tikz-cd}
\usepackage[capitalise]{cleveref}

\allowdisplaybreaks

\newtheorem{Theoremx}{Theorem}
 
\newtheorem{theorem}{Theorem}[section]
\theoremstyle{definition}
\newtheorem{setting}[theorem]{Setting}
\theoremstyle{definition}

\newtheorem{lemma}[theorem]{Lemma}

\newtheorem{proposition}[theorem]{Proposition}
\newtheorem{corollary}[theorem]{Corollary}

\newtheorem{question}{Question}

\theoremstyle{definition}
\newtheorem{definition}[theorem]{Definition}
\newtheorem{remark}[theorem]{Remark}
\theoremstyle{remark}

\newcommand{\Ass}{\operatorname{Ass}}

\newcommand{\Spec}{\operatorname{Spec}}

\usepackage{stmaryrd}

\newcommand{\Height}{\operatorname{ht}}

\newcommand{\id}{\operatorname{id}}

\newcommand{\Soc}{\operatorname{Soc}}
\newcommand{\Hom}{\operatorname{Hom}}

\newcommand{\depth}{\operatorname{depth}}

\newcommand{\Frac}{\operatorname{Frac}}

\newcommand{\Z}{\mathbb{Z}}

\newcommand{\sk}{\mathcal{k}}
\newcommand{\fm}{\mathfrak{m}}
\newcommand{\fp}{\mathfrak{p}}
\newcommand{\fq}{\mathfrak{q}}

\newcommand{\fn}{\mathfrak{n}}

\newcommand{\fP}{\mathfrak{P}}

\newcommand{\fQ}{\mathfrak{Q}}

\newcommand{\wh}{\widehat}
\makeatother

\crefname{Theoremx}{Theorem}{Theorems}

\usepackage{titlesec}		
\titleformat{\section}[block]{\large\scshape\bfseries\filcenter}{\thesection.}{1em}{}
\titleformat{\subsection}[hang]{\large\scshape\bfseries}{\thesubsection}{1em}{}	
\titleformat{\subsubsection}[hang]{\large\scshape\bfseries}{\thesubsubsection}{1em}{}

\usepackage[backend=biber, style=alphabetic, sorting=anyt,maxalphanames=5,maxbibnames=99]{biblatex}
\usepackage{hyperref}
\usepackage{xurl}
\hypersetup{breaklinks=true}

\renewbibmacro{in:}{}

\begin{document}

\title{Noncatenary splinters in prime characteristic}

\author[Loepp]{S. Loepp}
\address{Department of Mathematics and Statistics, Williams College, Williamstown, MA 01267}
\email{sloepp@williams.edu}

\author[Simpson]{Austyn Simpson}
\thanks{Simpson was supported by NSF postdoctoral fellowship DMS \#2202890.}
\address{Department of Mathematics, University of Michigan, Ann Arbor, MI 48109 USA}
\email{austyn@umich.edu}
\urladdr{\url{https://austynsimpson.github.io}}

\maketitle

\begin{abstract}
We construct a local Noetherian splinter (in fact, a weakly $F$-regular domain) in prime characteristic which is not catenary, which we view as an analogue of a theorem of Ogoma in equal characteristic zero. Moreover, we construct a weakly $F$-regular local UFD which is not Cohen--Macaulay. Both of these examples are obtained via finding sufficient conditions ensuring that a complete local ring of prime characteristic is the completion of some weakly $F$-regular local domain, which we expect to be of independent interest. 
\end{abstract}

\section{Introduction}

Let $(R,\fm,\sk)$ be a Noetherian local ring. Recall that $R$ is \emph{catenary} if given any pair of primes $\fp\subsetneq \fq\in\Spec R$, any two saturated chains of primes between $\fp$ and $\fq$ have the same length. The first example of a noncatenary ring was constructed by Nagata \cite{Nag56}, but the problem remained for several decades to find an example which was normal until Ogoma did just that \cite[\S III]{Ogo80} (see also \cite{Hei82}). Interpreting Ogoma's theorem through the lens of \emph{splinters} is the primary motivation for the present article.

Recall that a Noetherian local ring $(R,\fm,\sk)$ is a \emph{splinter} if any module-finite ring extension $\phi:R\rightarrow S$ splits in the category of $R$-modules --- that is, if there exists a map $\psi\in\Hom_R(S,R)$ such that $\psi\circ\phi=\id_R$. The content of the celebrated Direct Summand Theorem \cite{Hoc73,And18} is that regular local rings of any characteristic are splinters. Notably, the behavior of the splinter property is highly dependent on the characteristic of the ring. For example, one of the preeminent open conjectures in prime characteristic commutative algebra is that $F$-finite splinters are \emph{strongly $F$-regular}, a mild singularity type analogous to the class of Kawamata log terminal singularities appearing in the birational classification of complex algebraic varieties. By contrast, local rings of equal characteristic zero are splinters if and only if they are normal \cite[Lemma 2]{Hoc73}. From this point of view, Ogoma's theorem says that local splinters which contain the rational numbers can fail to be catenary, and our first contribution that we highlight is that the same phenomenon may occur in prime characteristic.

\begin{Theoremx}\label{maintheorem:noncatenary} (= \cref{theorem:noncatenary})
    There exists a local splinter of prime characteristic which is not catenary.
\end{Theoremx}
\cref{theorem:noncatenary} is \emph{a fortiori} stronger than the above statement in that the ring we construct turns out to be \emph{weakly $F$-regular}. This property (which implies the splinter condition --- see for instance \cite[Remark 2.4.1(2)]{DT23}) prescribes that Hochster and Huneke's operation of tight closure is trivial on all ideals of the ring. If $(R,\fm)$ is a local ring of prime characteristic $p>0$ and $I\subseteq R$ is an ideal, recall that the \emph{tight closure of $I$}, denoted $I^*$, is the ideal consisting of all ring elements $r\in R$ for which $cr^{p^e}\in I^{[p^e]}$ for some $c\in R\setminus\bigcup\limits_{\min(R)}\fp$ and all $e\gg 0$. Here, $I^{[p^e]}$ is the ideal generated by elements of the form $r^{p^e}$ where $r\in I$. We say that $R$ is weakly $F$-regular if $I=I^*$ for every ideal $I\subseteq R$.

The rings exhibiting \cref{maintheorem:noncatenary} are not excellent, as all excellent rings are catenary by definition. On this thread, it should be noted that both weak $F$-regularity and the splinter condition are considerably better behaved within the class of excellent rings, as is the operation of tight closure more broadly. Indeed, if $(R,\fm,\sk)$ is an excellent local ring of prime characteristic $p>0$, then tight closure enjoys \emph{colon capturing} --- that is, for a system of parameters $x_1,\ldots, x_d\in \fm$ one always has the containment\footnote{Heitmann has constructed non-excellent examples (UFDs, in fact) for which the containment (\ref{eq:cc}) fails \cite{Hei10}.}
\begin{equation}
(x_1,\ldots, x_{d-1}):x_d\subseteq (x_1,\ldots, x_{d-1})^*.\label{eq:cc}
\end{equation}

If $R$ is further assumed to be a splinter (resp. weakly $F$-regular) then the $\fm$-adic completion $\wh{R}$ is also a splinter \cite[Theorem C]{DT23} (resp. weakly $F$-regular \cite[Corollary 7.28]{HH94}). In the absence of excellence, both of these statements are false in general as first demonstrated in \cite{LR01} (see also \cite[Example 3.2.1]{DT23}). We obtain \cref{maintheorem:noncatenary} by exploiting this failure. Specifically, we find sufficient conditions ensuring that a complete local ring is the completion of \emph{some} weakly $F$-regular local domain, and these conditions are broad enough to allow for the examples described in \cref{maintheorem:noncatenary}. Our first result in this vein is:

\begin{Theoremx}\label{maintheorem:precompletion-1} (= \cref{theorem:f-regular-precompletion-non-gor,theorem:f-regular-precompletion-non-gor-noncatenary})
    Let $(T,\fm,\sk)$ be a complete local reduced (hence approximately Gorenstein) ring of prime characteristic $p>0$ with $\dim(T)\geq 2$. Suppose that $\{I_s\}$ is an approximately Gorenstein sequence of ideals, and let $y_s\in T$ generate the socle of $T/I_s$. Make the additional two assumptions:
\begin{enumerate}
    \item $y_s^{p^e}\not\in I_s^{[p^e]}$ for all $s$ and all $e$;\label{theorem:theoremA-1}
    \item there exists a nonmaximal prime $\fp\in\Spec(T)$ such that $M_{I_s,y_s}\subseteq \fp$ for all $s$.\label{theorem:theoremA-2}
\end{enumerate}
Then there exists a weakly $F$-regular local domain $(A,\fm\cap A)$ such that $\wh{A}\cong T$. Moreover, if we assume that $T$ is not equidimensional and has $\depth T\geq 2$ then $A$ may be chosen to be noncatenary.
\end{Theoremx}

The definition of the approximately Gorenstein property and of the ideals $M_{I_s,y_s}$ is postponed to \cref{section:non-gor}, but experts will recognize condition (\ref{theorem:theoremA-2}) as requiring (in case $T$ is $F$-finite and $F$-pure) that the \emph{splitting prime} of $T$ as defined in \cite{AE05} is not the maximal ideal (see also \cite[Lemma 6.2]{PT18}).

Notably absent from the list of assumptions on $T$ is the Cohen--Macaulay property, and it is this absence (among other details) which allows for the noncatenary ring described in \cref{maintheorem:noncatenary}. This points to another striking departure from the excellent scenario. Namely, if $(R,\fm)$ is a weakly $F$-regular local ring which is excellent (or more generally, is a homomorphic image of a Cohen-Macaulay local ring) then it follows that $R$ is Cohen--Macaulay\footnote{The same is true for excellent splinters, e.g. by \cite[Remark 2.4.1(3)]{DT23}.} by \cite[Proposition 4.2(c)]{HH94}. Since Cohen--Macaulay rings are catenary, \cref{maintheorem:noncatenary} also shows that there exist weakly $F$-regular rings which are not Cohen--Macaulay, which we believe to be novel in its own right. Using \cref{maintheorem:precompletion-1} we can even produce such examples which are unique factorization domains:

\begin{Theoremx}\label{maintheorem:noncm} (= \cref{corollary:non-cm-UFD})
    There exists a weakly $F$-regular local UFD which is not Cohen--Macaulay.
\end{Theoremx}

We can modify the setting of \cref{maintheorem:precompletion-1} slightly to ask which complete local rings (not necessarily UFDs) are completions of weakly $F$-regular local UFDs. Our methods for accomplishing this require us to additionally assume that $T$ is Gorenstein:

\begin{Theoremx}\label{maintheorem:precompletion-2} (= \cref{theorem:f-regular-precompletion-ufd})
     Let $(T,\fm)$ be a complete local reduced Gorenstein ring with $\dim T\geq 2$ of prime characteristic $p>0$. Let $I\subseteq T$ be a parameter ideal and let $y\in T$ generate the socle of $T/I$. Suppose that 
\begin{enumerate}
    \item $y^q\not\in I^{[q]}$ for all $q=p^e\gg 0$;
    \item there exists a nonmaximal prime $\fp\in\Spec T$ such that $y\not\in\fp$ and $M_{I,y}\subseteq\fp$.
\end{enumerate}
Then there exists an $F$-regular local UFD $(A,\fm\cap A)$ such that $\wh{A}\cong T$.
\end{Theoremx}

A few remarks are in order regarding the assumptions of \cref{maintheorem:precompletion-1,maintheorem:precompletion-2}. First, the reducedness assumption is necessary. Indeed, if $(A,\fn)$ is a (not necessarily excellent) weakly $F$-regular local ring, then $A$ is $F$-pure by \cite[Remark 1.6]{FW89}. It follows that $\wh{A}$ is $F$-pure and in particular is reduced. Additionally, in the situation above, since $\wh{A}$ is $F$-pure we know that $I=I^F$ for every ideal $I\subseteq A$. Here, $I^F$ denotes the Frobenius closure of $I$, and this equality of ideals forces condition (\ref{theorem:theoremA-1}) to hold in both theorems.

\cref{maintheorem:precompletion-1,maintheorem:precompletion-2} are generalizations of earlier work by the first named author and Rotthaus. For example, the statement of \cite[Theorem 23]{LR01} is similar to \cref{maintheorem:precompletion-2} in that $T$ is still assumed to be Gorenstein, but $T$ is required to be a domain and the techniques of \emph{op. cit.} are not sufficiently flexible to conclude that $A$ is a UFD.

\subsection*{Notational conventions and organization of the article}
In this article, a \emph{quasi-local} ring $(R,\fm)$ is taken to mean a commutative ring $R$ with unit and unique maximal ideal $\fm$. We say $R$ is \emph{local} if $R$ is quasi-local and Noetherian. $\min(R)$ refers to the set of minimal prime ideals of $R$, and $\Spec^\circ(R)$ denotes the punctured spectrum $\Spec(R)\setminus\{\fm\}$. Given a complete local ring $(T,\fm)$, we occasionally use the terminology \emph{precompletion of $T$} to refer to any local subring $(A,\fm\cap A)$ of $T$ such that the $\fm\cap A$-adic completion $\wh{A}\cong T$. We denote by $|B|$ the cardinality of a set $B$, and we say that $B$ is countable if it is either finite or in bijection with the natural numbers.

In \cref{section:non-gor} we prove the first statement of \cref{maintheorem:precompletion-1} and derive \cref{maintheorem:noncm} as a consequence. In \cref{section:gor} we prove \cref{maintheorem:precompletion-2}, followed by a specific example of a weakly $F$-regular UFD whose completion is not weakly $F$-regular using a construction of \cite{HRW97}. In \cref{section:noncatenary} we first prove \cref{maintheorem:noncatenary} using the results of \cref{section:non-gor}, and then finish the proof of \cref{maintheorem:precompletion-1}. We conclude the article with a list of questions in \cref{section:questions}.

\subsection*{Acknowledgments} We are grateful to Anurag Singh for helpful discussions concerning \cref{corollary:fermat-quintic}.

\section{Weakly \texorpdfstring{$F$}{F}-regular precompletions}\label{section:non-gor}
The goal of this section is to prove \cref{maintheorem:precompletion-1}, modulo the statement concerning (non)catenarity. Critical to our proof is Hochster's notion of an \emph{approximately Gorenstein ring}:

\begin{definition}\label{definition:approximately Gorenstein}
    A local ring $(R,\fm,\sk)$ is said to be \textit{approximately Gorenstein} if there exists a descending chain of $\fm$-primary ideals $I_1\supseteq I_2\supseteq\cdots\supseteq I_t\supseteq\cdots$ such that $R/I_t$ is a Gorenstein local ring for all $t$, and for every $N>0$, $I_t\subseteq \fm^N$ for all $t\gg 0$. We call such a chain $\{I_t\}$ an \textit{approximately Gorenstein sequence}.
\end{definition}

Note that the condition that the rings $R/I_t$ are Gorenstein is equivalent to $\dim_\sk \Soc(R/I_t)=1$. It is known that every excellent reduced ring is approximately Gorenstein, as is every local ring $R$ with $\depth(R)\geq 2$ \cite[Theorem 1.7]{Hoc77}. Moreover, a (not necessarily excellent) local ring $(R,\fm)$ is approximately Gorenstein if and only if $\wh{R}$ is \cite[Corollary 2.2(c)]{Hoc77}. In particular, any analytically unramified local ring is approximately Gorenstein.

Starting with a complete local approximately Gorenstein ring $T$ as in \cref{maintheorem:precompletion-1}, we will manufacture our weakly $F$-regular precompletion $A$ by ensuring that it contains the socle generators of the approximately Gorenstein sequence up to multiplying by a unit of $T$. We also ensure that $A$ contains no zerodivisors of $T$, which guarantees that $A$ is a domain. Finally, we construct $A$ so that it contains no nonzero elements of a given particular nonmaximal prime ideal $\fp$ of $T$. Constructing this ring $A$ is the goal of the next several results.

We start with the definiton of a $\fp$-subring of $T$, taken from \cite{Loe98}. Generally speaking, a $\fp$-subring of $T$ is a quasi-local subring of $T$ that is either countable or "small" compared to the residue field of $T$, contains no zerodivisors of $T$ and contains no nonzero elements of a chosen nonmaximal prime ideal $\fp$ of $T$.

\begin{definition} \cite[Definition 5]{Loe98}
Let $(T,\fm)$ be a complete local ring and let $(R,R \cap \fm)$ be a quasi-local subring of $T$.  Let $\fp \neq \fm$ be a prime ideal of $T$.  Suppose
\begin{enumerate}
\item $|R| \leq \sup(\aleph_0, |T/\fm|)$ with equality implying $T/\fm$ is countable, \label{item:p-subring-1}
\item $R \cap \fp = (0)$, and \label{item:p-subring-2}
\item $R \cap \fq = (0)$ for every $\fq \in \Ass(T)$.\label{item:p-subring-3}
\end{enumerate}
Then we call $R$ a \emph{$\fp$-subring of $T$}.
\end{definition}

Before we begin our construction, we state two useful lemmas showing that, under certain conditions, if $y$ is a nonunit of $T$ and $C$ is a set of nonmaximal prime ideals of $T$ with $y \not\in \fq$ for all $\fq \in C$, then we can find an associate of $y$ that avoids a selected set of cosets of the elements of $C$. We use \cref{lemma:countable-prime-avoidance} in the case that our $\fp$-subring is countable, and \cref{lemma:uncountable-prime-avoidance} in the case that it is not countable.

\begin{lemma}\label{lemma:countable-prime-avoidance} \cite[Lemma 16]{LR01}
Let $(T,\fm)$ be a complete local ring and let $C$ be a countable set of nonmaximal prime ideals of $T$.  Let $D$ be a countable set of elements of $T$.  Let $y \in \fm$ such that $y \not\in \fq$ for all $\fq \in C$.  Then there exists a unit $t$ of $T$ such that
$$yt \not\in \bigcup \{\fq + r \, | \,\fq \in C, \, r \in D\}.$$
\end{lemma}

\begin{lemma} (cf. \cite[Lemma 4]{Loe97})\label{lemma:uncountable-prime-avoidance}
Suppose $(T,\fm)$ is a local ring such that $T/\fm$ is infinite.  Let $C$ be a set of prime ideals of $T$ and let $y \in T$ such that $y \not\in \fq$ for all $\fq \in C$.  Let $D$ be a subset of $T$.  If $|C \times D| < |T/\fm|$, then there exists a unit $t$ of $T$ such that
$$yt \not\in \bigcup \{\fq + r \, | \,\fq \in C, \, r \in D\}.$$
\end{lemma}

To construct $A$, we build an increasing chain of $\fp$-subrings of $T$. To get from one $\fp$-subring to the next one, we adjoin a carefully chosen element of $T$ such that the resulting subring of $T$ itself is a $\fp$-subring of $T$. \cref{lemma:transcendental} gives sufficient conditions on an element $x$ of $T$ so that, if $R$ is a $\fp$-subring of $T$, then $R[x]_{(R[x] \cap \fm)}$ is also a $\fp$-subring of $T$. Later, we use \cref{lemma:countable-prime-avoidance,lemma:uncountable-prime-avoidance} to find elements of $T$ that satisfy these sufficient conditions.

\begin{lemma}\label{lemma:transcendental}
Let $(T,\fm)$ be a complete local ring and let $\fp$ be a nonmaximal prime ideal of $T$. Let $(R,R \cap \fm)$ be a $\fp$-subring of $T$ and let $C = \{\fp\} \cup \Ass(T)$.  If for every $\fq \in C$ we have that $x + \fq \in T/\fq$ is transcendental over $R/(R \cap \fq) \cong R$, then $S = R[x]_{(R[x] \cap \fm)}$ is a $\fp$-subring of $T$.  Moreover, if $R$ is finite, then $S$ is countably infinite, while if $R$ is infinite, then $|S|= |R|.$
\end{lemma}

\begin{proof}
Note that $x$ is transcendental over $R$ and so if $R$ is finite, then $S$ is countably infinite, while if $R$ is infinite, then $|S| = |R|$. Let $\fq \in C$ and suppose $f \in R[x] \cap \fq$.  Then $f = r_n(x)^n +  \cdots + r_1(x) + r_0 \in \fq$ where $r_i \in R$ for all $i \in \{1,2, \ldots,n\}$.  Since $x + \fq \in T/\fq$ is transcendental over $R/(R \cap \fq) \cong R$, we have that $r_i \in R \cap \fq = (0)$ for all $i \in \{1,2, \ldots,n\}$, and it follows that $S \cap \fq = (0)$ for all $\fq \in C$. Thus, $S$ is a $\fp$-subring of $T$.
\end{proof}

\cref{lemma:yt-in-ring} should be compared with \cite[Lemma 18]{LR01}, which further assumes that $T$ is a domain. In fact, the statement of the lemma, as well as the proof are based on \cite[Lemma 18]{LR01}. The importance of \cref{lemma:yt-in-ring} is that it shows, given a particular element $y$ of $T$, there exists a $\fp$-subring of $T$ that contains an associate of $y$. We use the result repeatedly to ensure that our ring $A$ contains the desired socle generators (up to a unit) of an approximately Gorenstein sequence.

\begin{lemma}\label{lemma:yt-in-ring}
Let $(T,\fm)$ be a complete local ring and let $\fp$ be a nonmaximal prime ideal of $T$.  Let $y \in \fm \setminus \fp$ be a regular element of $T$, and let $(R, R \cap \fm)$ be a $\fp$-subring of $T$.  Then there exists a $\fp$-subring $(S, S \cap \fm)$ of $T$ such that $R \subseteq S \subseteq T$ and $yt \in S$ for some unit $t$ of $T$.  Moreover, if $R$ is finite, then $S$ is countably infinite, while if $R$ is infinite, then $|S| = |R|$.
\end{lemma}

\begin{proof}
Define $C = \{\fp\} \cup \Ass(T)$, and note that $C$ is a finite set with $y \not\in \fq$ for all $\fq \in C$.  Let $\fq \in C$ and define $D_{(\fq)} \subset T$ to be a full set of coset representatives of the elements $u + \fq$ of $T/\fq$ that are algebraic over $R/(R \cap \fq) \cong R$.  Define $D = \bigcup_{\fq \in C}D_{(\fq)}$ and note that $|D| \leq |R|.$ Now use \cref{lemma:countable-prime-avoidance} if $R$ is countable and \cref{lemma:uncountable-prime-avoidance} if not to find a unit $t$ of $T$ such that $yt \not\in \bigcup \{\fq + r \, | \,\fq \in C, \, r \in D\}.$ Define $S = R[yt]_{(R[yt] \cap \fm)}$ and note that $R \subseteq S \subseteq T$ and $yt \in S$. By our choice of $t$, $yt + \fq \in T/\fq$ is transcendental over $R/(R \cap \fq) \cong R$ for all $\fq \in C$. It follows by \cref{lemma:transcendental} that $S$ is a $\fp$-subring of $T$ and that if $R$ is finite, then $S$ is countably infinite, while if $R$ is infinite, then $|S| = |R|$. 
\end{proof}

Given an ascending chain of $\fp$-subrings, it will prove useful to be able to conclude that the union of the chain is also a $\fp$-subring. Our next result gives conditions under which this holds. We begin with a useful definition.

\begin{definition}
Let $\Omega$ be a well-ordered set and let $\alpha \in \Omega$.  We define $\gamma(\alpha) = \sup\{\beta \in \Omega \, | \, \beta < \alpha\}$.
\end{definition}

\begin{lemma}\label{lemma:unioning}
Let $(T,\fm)$ be a complete local ring and let $\fp$ be a nonmaximal prime ideal of $T$.  Let $\Omega$ be a well-ordered index set with least element $1$ such that either $\Omega$ is countable or, for every $\alpha \in \Omega$, we have $|\{\beta \in \Omega \, | \, \beta < \alpha \}| < |T/\fm|$. Let $\{(R_{\beta}, R_{\beta} \cap \fm)\}_{\beta \in \Omega}$ be an ascending collection of infinite quasi-local subrings of $T$ such that $R_1$ is a $\fp$-subring of $T$, if $\gamma(\alpha) < \alpha$ then $R_{\alpha}$ is a $\fp$-subring of $T$ with $|R_{\alpha}| = |R_{\gamma(\alpha)}|$ and if $\gamma(\alpha) = \alpha$ then $R_{\alpha} = \bigcup_{\beta < \alpha}R_{\beta}$. Then $S = \bigcup_{\beta \in \Omega} R_{\beta}$ satisfies the conditions to be a $\fp$-subring of $T$ except for possibly the cardinality condition (\ref{item:p-subring-1}), and moreover satisfies the inequality $|S| \leq \sup(|R_1|, |\Omega|).$ 
\end{lemma}

\begin{proof}
First note that if $\fq \in \{\fp\} \bigcup \Ass(T)$, then $S \cap \fq = (0)$ follows since $R_{\beta} \cap \fq = (0)$ for all $\beta \in \Omega$. The inequality $|S| \leq \sup(|R_1|, |\Omega|)$ follows from the proof of \cite[Lemma 6]{Hei93}.
\end{proof}

\cref{lemma:base-ring} is crucial in showing that our final ring $A$ contains the socle generators of the approximately Gorenstein sequence (up to a unit).  To prove the lemma, we repeatedly apply \cref{lemma:yt-in-ring}.

\begin{lemma}\label{lemma:base-ring}
Let $(T,\fm)$ be a complete local ring and let $\Pi$ be the prime subring of $T$.  Suppose that all elements of $\Pi$ are regular elements of $T$ and let $\fp$ be a nonmaximal prime ideal of $T$ such that $\Pi \cap \fp = (0)$.  Let $y_1, y_2, \ldots$ be a countable collection of regular elements of $T$ with $y_i \in \fm \setminus \fp$ for all $i = 1,2, \ldots$.  Then there exists a countably infinite $\fp$-subring $(R, R \cap \fm)$ of $T$ such that, for every $i = 1,2, \ldots$, there is a unit $t_i$ of $T$ satisfying $y_it_i \in R$.  
\end{lemma}

\begin{proof}
Define $R_0$ to be $\Pi_{(\Pi \cap \fm)}$ and note that, by our hypotheses, $R_0$ is a $\fp$-subring of $T$.  Now use \cref{lemma:yt-in-ring} to find a countably infinite $\fp$-subring $(R_1, R_1 \cap \fm)$ of $T$ such that $R_0 \subseteq R_1 \subseteq T$ and $y_1t_1 \in R_1$ for some unit $t_1$ of $T$. Again use \cref{lemma:yt-in-ring} to find a countably infinite $\fp$-subring $(R_2, R_2 \cap \fm)$ of $T$ such that $R_1 \subseteq R_2 \subseteq T$ and $y_2t_2 \in R_2$ for some unit $t_2$ of $T$.  Continue inductively to define $R_n$ for every $n = 1,2, \ldots,$ so that $(R_n, R_n \cap \fm)$ is a countably infinite $\fp$-subring of $T$ satisfying $R_{n-1} \subseteq R_n \subseteq T$ and $y_nt_n \in R_n$ for some unit $t_n$ of $T$.  Define $R = \bigcup_{i = 1}^{\infty}R_i$.  By \cref{lemma:unioning}, $R$ is a $\fp$-subring of $T$ and by construction, for every $i = 1,2, \ldots$, there is a unit $t_i$ of $T$ satisfying $y_it_i \in R$.  
\end{proof}

To construct the domain $A$, we start with the $\fp$-subring guaranteed to exist by \cref{lemma:base-ring}, which has the effect of ensuring that our final ring $A$ will contain associates of the socle generators of the approximately Gorenstein sequence. This ring, however, may not have the desired completion.  We next adjoin elements to our starting $\fp$-subring of $T$, maintaining the $\fp$-subring properties, until the ring has the desired completion. The next result, taken from \cite{Hei94}, gives sufficient conditions for a subring of a complete local ring $T$ to have completion isomorphic to $T$. 

\begin{proposition}\label{proposition:heitmann-completion} \cite[Proposition 1]{Hei94}
Let $(R, R \cap \fm)$ be a quasi-local subring of a complete local ring $(T,\fm)$ such that the map $R \longrightarrow T/\fm^2$ is onto and $IT \cap R = I$ for every finitely generated ideal $I$ of $R$. Then $R$ is Noetherian and the natural homomorphism $\wh{R} \longrightarrow T$ is an isomorphism.
\end{proposition}

We now work to construct our ring $A$ to satisfy the conditions of \cref{proposition:heitmann-completion}. The following two lemmas will be useful for that purpose. We note that both lemmas can be thought of as generalizations of the prime avoidance theorem.

\begin{lemma}\label{lemma:countable}\cite[Lemma 2]{Hei93}
Let $(T,\fm)$ be a complete local ring and let $C$ be a countable set of nonmaximal prime ideals of $T$.  Let $D$ be a countable set of elements of $T$.  If $I\subseteq T$ is an ideal which is contained in no single $\fq\in C$, then $I \not\subseteq \bigcup \{\fq + r \, | \, \fq \in C, r \in D\}$.
\end{lemma}

\begin{lemma}\label{lemma:uncountable}\cite[Lemma 3]{Hei93}
Let $(T,\fm)$ be a local ring and let $C$ be a set of prime ideals of $T$.  Let $D$ be a set of elements of $T$.  Suppose $|C \times D| < |T/\fm|$.  If $I\subseteq T$ is an ideal which is contained in no single $\fq \in C$, then $I \not\subseteq \bigcup \{\fq + r \, | \, \fq \in C, r \in D\}$.
\end{lemma}

We next show that we can adjoin an element of an arbitrary coset $u + \fm^2$ without disturbing the $\fp$-subring properties. We use the result to construct $A$ so that the map $A \longrightarrow T/\fm^2$ is onto, which is a needed condition for \cref{proposition:heitmann-completion}. \cref{lemma:adjoin-coset} should be compared with \cite[Lemma 19]{LR01} where $T$ is further assumed to be a domain.

\begin{lemma}\label{lemma:adjoin-coset}
Let $(T,\fm)$ be a complete local ring with $\fm \not\in \Ass(T)$ and let $\fp$ be a nonmaximal prime ideal of $T$. Let $(R,R \cap \fm)$ be an infinite $\fp$-subring of $T$ and let $u \in T$.  Then there exists a $\fp$-subring $(S, S \cap \fm)$ of $T$ such that $R \subseteq S \subseteq T$, $u + \fm^2$ is in the image of the map $S \longrightarrow T/\fm^2$, and $|S| = |R|$.
\end{lemma}

\begin{proof}
Define $C = \{\fp\} \cup \Ass(T)$ and let $\fq \in C$.  Let $D_{(\fq)}$ be a full set of coset representatives of the elements $z + \fq \in T/\fq$ that make $z + u + \fq$ algebraic over $R$, and define $D = \bigcup_{\fq \in C}D_{(\fq)}$.  Since $\fm \not\subseteq \fq$ for all $\fq \in C$, we have $\fm^2 \not\subseteq \fq$ for all $\fq \in C$. Use \cref{lemma:countable} if $R$ is countable and \cref{lemma:uncountable} if not, both with $I = \fm^2$, to find $x \in \fm^2$ such that $x \not\in \bigcup \{\fq + r \, | \,\fq \in C, \, r \in D\}.$  Define $S = R[x + u]_{(R[x + u] \cap \fm)}$.  Note that $R \subseteq S \subseteq T$ and $u + \fm^2$ is in the image of the map $S \longrightarrow T/\fm^2$. By \cref{lemma:transcendental}, $S$ is a $\fp$-subring of $T$ and $|S| = |R|$.
\end{proof}

To use \cref{proposition:heitmann-completion} it will be crucial to ensure that our final ring $A$ satisfies the condition that $IT \cap A = I$ for every finitely generated ideal $I$ of $A$. The next two lemmas are in service of this.

\begin{lemma}\label{lemma:c-in-IS}
Let $(T,\fm)$ be a complete local ring and let $\fp$ be a nonmaximal prime ideal of $T$. Let $(R,R \cap \fm)$ be a $\fp$-subring of $T$.  Suppose $I$ is a finitely generated ideal of $R$ and let $c \in IT \cap R$.  Then there exists a $\fp$-subring $(S, S \cap \fm)$ of $T$ such that $R \subseteq S \subseteq T$ and $c \in IS$.  Moreover, if $R$ is infinite, then $|S| = |R|$.
\end{lemma}
\begin{proof}
    The first statement regarding the existence of the desired $\fp$-subring $S$ is \cite[Lemma 6]{Loe98}. Since $S$ is constructed from $R$ by adjoining finitely many elements from $T$ and then localizing, the second statement holds.
\end{proof}

\begin{lemma} (cf. \cite[Lemma 21]{LR01})\label{lemma:intermediate-p-subring}
Let $(T,\fm)$ be a complete local ring with $\fm \not\in \Ass(T)$ and let $\fp$ be a nonmaximal prime ideal of $T$. Let $(R,R \cap \fm)$ be an infinite $\fp$-subring of $T$ and let $u \in T$.  Then there exists a $\fp$-subring $(S, S \cap \fm)$ of $T$ such that\begin{enumerate}
    \item $R \subseteq S \subseteq T$,
    \item $u + \fm^2$ is in the image of the map $S \longrightarrow T/\fm^2$,
    \item For every finitely generated ideal $I$ of $S$, we have $IT \cap S = I$, and
    \item $|S| = |R|$.
\end{enumerate}
\end{lemma}

\begin{proof}
First use \cref{lemma:adjoin-coset} to find a $\fp$-subring $(R_1, R_1 \cap \fm)$ of $T$ such that $R \subseteq R_1 \subseteq T$, $u + \fm^2$ is in the image of the map $R_1 \longrightarrow T/\fm^2$, and $|R_1| = |R|$.  Let $$\Omega = \{(I,c) \, | \, I \mbox{ is a finitely generated ideal of } R_1 \mbox{ and } c \in IT \cap R_1\}.$$
Note that $|\Omega| = |R_1|$.  Well-order $\Omega$ so that $1$ is its least element and so that it does not have a maximal element.  We inductively define an ascending collection $\{R_{\alpha}\}_{\alpha \in \Omega}$ of $\fp$-subrings of $T$.  We have already defined $R_1$.  Let $\alpha \in \Omega$ and assume that $R_{\beta}$ has been defined for all $\beta < \alpha$.  If $\gamma(\alpha) < \alpha$ and $\gamma(\alpha) = (I,c)$ then define $R_{\alpha}$ to be the $\fp$-subring of $T$ obtained from \cref{lemma:c-in-IS} so that $R_{\gamma(\alpha)} \subseteq R_{\alpha} \subseteq T$, $c \in IR_{\alpha}$, and $|R_{\alpha}| = |R_{\gamma(\alpha)}|$. If, on the other hand, $\gamma(\alpha) = \alpha$, then define $R_{\alpha} = \bigcup_{\beta < \alpha}R_{\beta}$. By \cref{lemma:unioning}, $R_2 = \bigcup_{\alpha \in \Omega}R_{\alpha}$ is a $\fp$-subring of $T$ and $|R_2| = |R_1|$. Now let $I$ be a finitely generated ideal of $R_1$ and let $c \in IT \cap R_1$.  Then $(I,c) = \gamma(\alpha)$ for some $\alpha \in \Omega$ with $\gamma(\alpha) < \alpha$. By construction, $c \in IR_{\alpha} \subseteq IR_2$.  It follows that $IT \cap R_1 \subseteq IR_2$ for every finitely generated ideal $I$ of $R_1$.

Repeat the above construction replacing $R_1$ with $R_2$ to find a $\fp$-subring $R_3$ of $T$ such that $R_2 \subseteq R_3 \subseteq T$, $|R_3| = |R_2|$, and $IT \cap R_2 \subseteq IR_3$ for every finitely generated ideal $I$ of $R_2$.  Continue to obtain a chain of $\fp$-subrings $R_1 \subseteq R_2 \subseteq R_3 \subseteq \cdots$ of $T$ such that, for all $n$, $|R_n| = |R|$, and, for every finitely generated ideal $I$ of $R_n$, we have $IT \cap R_n \subseteq R_{n + 1}$. Let $S = \bigcup_{i = 1}^{\infty}R_i$. By \cref{lemma:unioning}, $S$ is a $\fp$-subring of $T$ and $|S| = |R|$. Now suppose that $I\subseteq S$ is a finitely generated ideal.  Write $I = (a_1, \ldots ,a_m)$ for some $a_i \in S$. Let $c \in IT \cap S$. Then there is an $N$ such that $c \in R_N$ and $a_i \in R_N$ for all $i \in \{1,2, \ldots ,m\}$. Now, $c \in (a_1,a_2, \ldots,a_m)T \cap R_N \subseteq (a_1,a_2, \ldots,a_m)R_{N+1} \subseteq I$. Hence $IT \cap S = I$, and it follows that $S$ is the desired $\fp$-subring of $T$.
\end{proof}

We now have all of the tools needed to construct our final ring $A$.

\begin{theorem}\label{theorem:p-subring-precompletion}
Let $(T,\fm)$ be a complete local ring and let $\Pi$ be the prime subring of $T$.  Suppose that all elements of $\Pi$ are regular elements of $T$ and let $\fp$ be a nonmaximal prime ideal of $T$ such that $\Pi \cap \fp = (0)$.  Let $y_1, y_2, \ldots$ be a countable collection of regular elements of $T$ with $y_i \in \fm \setminus \fp$ for all $i = 1,2, \ldots$.  Then there exists a local domain $A$ such that $\wh{A} \cong T$, $A \cap \fp = (0)$, and, for every $i = 1,2, \ldots$, there is a unit $t_i$ of $T$ satisfying $y_it_i \in A$.  
\end{theorem}

\begin{proof}
Let $\Omega = T/\fm^2$, and well-order $\Omega$ so that every element of $\Omega$ has fewer than $|\Omega$| predecessors.  Let $1$ denote the least element of $\Omega$. Let $(R_1, R_1 \cap \fm)$ be the countably infinite $\fp$-subring of $T$ given by \cref{lemma:base-ring}. So, for every $i = 1,2, \ldots,$ $R_1$ contains $y_it_i$ for some unit $t_i \in T$.  We now recursively define an ascending collection $\{R_{\alpha}\}_{\alpha \in \Omega}$ of infinite quasi-local subrings of $T$. Note that $R_1$ has already been defined. 
 Let $\alpha \in \Omega$ and assume that $R_{\beta}$ has been defined for every $\beta < \alpha$.  If $\gamma(\alpha) < \alpha$, define $R_{\alpha}$ to be the $\fp$-subring obtained from \cref{lemma:intermediate-p-subring} so that $R_{\gamma(\alpha)} \subseteq R_{\alpha} \subseteq T$, $\gamma(\alpha)$ is in the image of the map $R_{\alpha} \longrightarrow T/\fm^2$, $IT \cap R_{\alpha} = I$ for every finitely generated ideal $I$ of $R_{\alpha}$, and $|R_{\alpha}| = |R_{\gamma(\alpha)}|$.  If $\gamma(\alpha) = \alpha$, define $R_{\alpha} = \bigcup_{\beta < \alpha}R_{\beta}$.  Now define $A = \bigcup_{\alpha \in \Omega}R_{\alpha}$.  By \cref{lemma:unioning}, $A \cap \fp = (0)$ and if $\fq \in \Ass(T)$ then $A \cap \fq = (0)$. By construction, the map $A \longrightarrow T/\fm^2$ is onto. Let $I\subseteq A$ be a finitely generated ideal and let $c \in IT \cap A$.  Write $I = (a_1, a_2, \ldots,a_m)$ for $a_i \in A$.  Choose $\alpha$ so that $\gamma(\alpha) < \alpha$ and $c, a_1, a_2, \ldots ,a_m \in R_{\gamma(\alpha)}$. Now, $c \in (a_1,a_2, \ldots ,a_m)T \cap R_{\alpha} = (a_1, a_2, \ldots,a_m)R_{\alpha} \subseteq I$.  It follows that $IT \cap A = I$ for all finitely generated ideals $I\subseteq A$. By \cref{proposition:heitmann-completion}, $\wh{A} \cong T$, and since $\fq \cap A = (0)$ for all $\fq \in \Ass(T)$, $A$ is a domain.
\end{proof}

Having developed the necessary completion machinery, we can now proceed with our investigation of weakly $F$-regular precompletions. We recall the following characterization of weak $F$-regularity in terms of an approximately Gorenstein sequence.

\begin{lemma} \cite[Proposition 8.8]{HH90}\label{lemma:tc-approx-gor}
    Let $(R,\fm,\sk)$ be a local approximately Gorenstein ring of prime characteristic $p>0$ with approximately Gorenstein sequence $\{I_s\}$. Then $R$ is weakly $F$-regular if and only if $I_s=I_s^*$ for all $t$.
\end{lemma}

Here and in the sequel we will employ the ideals $M_{I,y}$ from \cite{LR01}. For an ideal $I\subseteq R$ and an element $y\in R$, we define

\begin{align*}
    M_{I,y}:=\bigcup\limits_{e\geq 1}\bigcap\limits_{f\geq e} (I^{[p^f]}:y^{p^f}).
\end{align*}

\begin{theorem}\label{theorem:f-regular-precompletion-non-gor}
    Let $(T,\fm,\sk)$ be a complete local reduced (approximately Gorenstein) ring of prime characteristic $p>0$ and $\dim(T)\geq 2$. Suppose that $\{I_s\}$ is an approximately Gorenstein sequence of ideals, and let $y_s\in T$ generate the socle of $T/I_s$. Make the additional two assumptions:
\begin{enumerate}
    \item $y_s^{p^e}\not\in I_s^{[p^e]}$ for all $s$ and all $e$;
    \item there exists a nonmaximal prime $\fp\in\Spec^\circ(T)$ such that for all $s$, $M_{I_s,y_s}\subseteq \fp$.
\end{enumerate}
Then there exists a weakly $F$-regular local domain $(A,\fm\cap A)$ such that $\wh{A}\cong T$. Moreover, given countably many regular elements $v_1,v_2,\ldots\in \fm \setminus \fp$, the ring $A$ may be chosen to contain a $T$-associate of each $v_i$.
\end{theorem}
\begin{proof}
Let $\min(T)=\{\fq_1,\dots, \fq_n\}$ be the minimal prime ideals of $T$ and fix a natural number $s$. Since $I_s$ is $\fm$-primary and $\fp$ is nonmaximal, we know that $I_s+(y_s)\nsubseteq\fp\cup\bigcup\limits_{j=1}^n\fq_j$. Hence, by prime avoidance, there exists an element $z_s\in I_s$ such that $y_s+z_s\not\in\fq_j$ for all $j$ and $y_s+z_s\not\in\fp$. Note also that $M_{I_s,y_s}=M_{I_s,y_s+z_s}$ and $(y_s+z_s)^{p^e}\not\in I_s^{[p^e]}$ for all $e$ by assumption. Since $T$ is reduced, after replacing $y_s$ with $y_s+z_s$ for each $s$, we may assume that each $y_s$ is a regular element of $T$.

We may now construct a local domain $(A,\fm\cap A)$ via \cref{theorem:p-subring-precompletion} such that $\fp\cap A=(0)$, $\wh{A}\cong T$, and such that for every $i$ there exists a unit $t_i\in T$ where $y_it_i\in A$. Let $\{J_i\}$ be the corresponding approximately Gorenstein sequence for $A$ such that $\wh{J_i}=I_i$. We claim that $A$ is weakly $F$-regular, and it suffices by \cref{lemma:tc-approx-gor} to show that $J_s=J_s^*$ for all $s$. Fix $s$ and note that $y_st_s$ generates the socle of $A/J_s$. It suffices to show that $y_st_s\not\in J_s^*$. If not, there exists $0\neq c\in A$ so that $c(y_st_s)^{p^e}\in J_s^{[p^e]}$ for all $e$. It follows that $c\in M_{I_s,y_st_s}\cap A=M_{I_s,y_s}\cap A\subseteq \fp\cap A=(0)$, a contradiction. Applying \cref{theorem:p-subring-precompletion} to $\{v_i\}\cup\{y_i\}$ in place of $\{y_i\}$ above does not affect the weak $F$-regularity of $A$, so we immediately obtain the second statement as well.
\end{proof}
We are able to extend the previous result by adjoining a formal power series variable via the following:
\begin{proposition}\cite[Proposition 3]{LR01}\label{proposition:multiplier}
    Let $(A,\fm)$ be a reduced local ring of prime characteristic $p>0$, let $I\subseteq A\llbracket t\rrbracket$ be an ideal, and let $x\in A\llbracket t\rrbracket$. Then either $M_{I,x}=A\llbracket t\rrbracket$ or $M_{I,x}\subseteq \fm A\llbracket t\rrbracket$.
\end{proposition}

\begin{corollary}\label{corollary:f-regular-power-series-precompletion-non-gor}
    Let $(T,\fm)$ be a complete local reduced (approximately Gorenstein) ring of prime characteristic $p>0$ with $\dim T\geq 1$. Suppose that $\{I_s\}$ is an approximately Gorenstein sequence of ideals, and let $y_s\in T$ generate the socle of $T/I_s$. If $y_s\not\in I_s^{[p^e]}$ for all $s$ and all $e$ (for example, we may take $T$ to be $F$-pure), then there exists a weakly $F$-regular local domain $A$ such that $\wh{A}\cong T\llbracket X\rrbracket$. Moreover, given countably many regular elements $v_1,v_2,\ldots\in (\fm,X)\setminus \fm T\llbracket X\rrbracket$, we may choose $A$ to contain a $T\llbracket X\rrbracket$-associate of each $v_i$.
\end{corollary}

\begin{proof}
     For each $s$, let $L_s:=(I_s,X^s)$. Note that $\{L_s\}$ is an approximately Gorenstein sequence for $T\llbracket X\rrbracket$ and $z_s:=y_s+X^s$ generates the socle of $T\llbracket X\rrbracket/L_s$. By our assumptions we know that $z_s^{p^e}=y_s^{p^e}+X^{sp^e}\not\in(I_s^{[p^e]}, X^{sp^e})= L_s^{[p^e]}$. Let $\fp:=\fm T\llbracket X\rrbracket$ and note that $z_s\not\in\fp$ for all $s$. By \cref{proposition:multiplier}, we know that for any fixed $s$, either $M_{L_s,z_s}=T\llbracket X\rrbracket$ or $M_{L_s,z_s}\subseteq \fp$. However, it must be the latter because
    \begin{align}
        M_{L_s,z_s}&=\bigcup\limits_{e\geq 1}\bigcap\limits_{f\geq e}\left(L_s^{[p^f]}:_{T\llbracket X\rrbracket} z_s^{p^f}\right)\nonumber\\
        & =\bigcup\limits_{e\geq 1}\bigcap\limits_{f\geq e}\left(\left(I_s^{[p^f]},X^{sp^f}\right):_{T\llbracket X\rrbracket} (y_s^{p^f}+X^{sp^f})\right)\nonumber\\
        & = \bigcup\limits_{e\geq 1}\bigcap\limits_{f\geq e}\left (I_s^{[p^f]}T\llbracket X\rrbracket:_{T\llbracket X\rrbracket} y_s^{p^f}\right)\nonumber\\
        & = \bigcup\limits_{e\geq 1}\bigcap\limits_{f\geq e} (I_s^{[p^f]}:_T y_s^{p^f})T\llbracket X\rrbracket\label{eq:power-series-1}\\
        & = M_{I_s T\llbracket X\rrbracket, y_s}\nonumber\\
        &\neq T\llbracket X\rrbracket\label{eq:power-series-2}
    \end{align}
    where (\ref{eq:power-series-1}) follows from flatness of $T\rightarrow T\llbracket X\rrbracket$ and (\ref{eq:power-series-2}) from the fact that $y_s^{p^e}\not\in I_s^{[p^e]}$. We may now apply \cref{theorem:f-regular-precompletion-non-gor} to $\fp\subseteq T\llbracket X\rrbracket$ to obtain both conclusions simultaneously.
\end{proof}

\subsection{A weakly \texorpdfstring{$F$}{F}-regular UFD which is not Cohen--Macaulay}
We are now able to deduce \cref{maintheorem:noncm} from \cref{corollary:f-regular-power-series-precompletion-non-gor}.
\begin{theorem}\label{corollary:non-cm-UFD}
    There exists a weakly $F$-regular local UFD $A$ which is not Cohen--Macaulay.
\end{theorem}
\begin{proof}
    Let $T$ be any complete local UFD with $\depth T\geq 3$ which is $F$-pure and not Cohen--Macaulay. Since $T$ is reduced and complete, $T$ is approximately Gorenstein and so the conditions of \cref{corollary:f-regular-power-series-precompletion-non-gor} apply to give a weakly $F$-regular local domain $A$ such that $\wh{A}\cong T\llbracket X\rrbracket$. Moreover, $T\llbracket X\rrbracket$ is a (non-Cohen--Macaulay) UFD by \cite[Satz 2]{Sch67}, hence so is $A$ by \cite[Ch. VII, \S3.6, Proposition 4]{Bou89}, as desired.

    We can produce specific examples of $T$ with the above properties as follows. Let $\sk$ be a field of characteristic $2$ and let $R=\sk[x_1,x_2,x_3,x_4]^{\Z/4\Z}$ be the invariant subring of $\sk[x_1,x_2,x_3,x_4]$ with respect to the action which cyclically permutes the variables. Let $\fn=(x_1,x_2,x_3,x_4)\cap R$, and finally let $T:=\wh{R_\fn}$. It is known that $R$ is a non-Cohen--Macaulay UFD \cite{Ber67} which is $F$-pure \cite[Proposition 2.4(a)]{Gla95}, and the same is true for $T$ by \cite{FG75}. The computation for $\depth T=3$ is handled in \cite[Ch. VI, Proposition 2.3]{AF78}.
\end{proof}

\section {\texorpdfstring{$F$}{F}-regular UFD precompletions of a complete local Gorenstein ring}\label{section:gor}
The goal of this section is to prove \cref{maintheorem:precompletion-2}. Analogous to the previous section, the ring $A$ which we produce will be $F$-regular. The improvement is that $A$ will also be a UFD, but the compromise is that we will require $T$ to be Gorenstein. Recall that a ring $R$ is \emph{$F$-regular} if $S^{-1}R$ is weakly $F$-regular for all multiplicative subsets $S\subseteq R$. We do not distinguish in this section between $F$-regularity and weak $F$-regularity because the two notions coincide in the Gorenstein setting \cite[Corollary 4.7(a)]{HH94}.

The construction of our UFD $A$ is similar to the construction of $A$ in the previous section. One difference is that in this section, we use the concept of N-subrings of $T$ in place of $\fp$-subrings from the previous section.

\begin{definition} \cite{Hei93}
Let $(T,\fm)$ be a complete local ring and let $(R, R \cap \fm)$ be a quasi-local unique factorization domain contained in $T$ satisfying
\begin{enumerate}
\item $|R| \leq \sup(\aleph_0, |T/\fm|)$ with equality only if $T/\fm$ is countable,\label{item:N-subring-1}
\item $R \cap Q = (0)$ for all $Q \in \Ass(T)$, and\label{item:N-subring-2}
\item if $t \in T$ is regular and $J \in \Ass(T/tT)$, then ht$(R \cap J) \leq 1$. \label{item:N-subring-3}
\end{enumerate}
Then $R$ is called an \emph{N-subring of $T$}.
\end{definition}

The following result is  analogous to \cref{lemma:transcendental} from the previous section. That is, the result gives sufficient conditions on an element $x$ of $T$ so that, given an N-subring $R$ of $T$, the ring $R[x]_{(R[x] \cap \fm)}$ is an N-subring.

\begin{lemma}\label{lemma:N-subring-transcendental}\cite[Lemma 11]{Loe97}
Let $(T,\fm)$ be a complete local ring and $\fp \in \Spec(T)$. Let $(R,R \cap \fm)$ be an N-subring of $T$ with $R \cap \fp = (0)$.  Suppose $C \subseteq \Spec(T)$ satisfies the following conditions:
\begin{enumerate}
\item $\fm \not\in C$,
\item $\fp \in C$,
\item $\{\fq \in \Spec(T) \, | \, \fq \in \Ass(T/rT) \mbox{ with } 0 \neq r \in R\} \subseteq C$, and
\item $\Ass(T) \subseteq C$.
\end{enumerate}
Let $x \in T$ be such that $x \not\in \fq$ and $x + \fq \in T/\fq$ is transcendental over $R/(R \cap \fq)$ for every $\fq \in C$.  Then $S = R[x]_{(R[x] \cap \fm)}$ is an N-subring of $T$ properly containing $R$, $|S| = \sup(\aleph_0, |R|)$, and $S \cap \fp = (0)$.
\end{lemma}

We now show that, given an element $y\in T$ satisfying some specific properties, we can find an appropriate N-subring of $T$ that contains an associate of $y$.

\begin{lemma}\label{lemma:N-subring-base-ring}
Let $(T,\fm)$ be a complete local equicharacteristic ring with $\depth T \geq 1$, and let $\fp$ be a nonmaximal prime ideal of $T$.  Let $y$ be a regular element of $T$ with $y \in \fm \setminus \fp$.  Then there exists a countably infinite N-subring $(R, R \cap \fm)$ of $T$ such that $R \cap \fp = (0)$ and $yt \in R$ for some unit $t \in T$. 
\end{lemma}

\begin{proof}
Let $\Pi$ be the prime subring of $T$ and define $R_0$ to be $\Pi_{(\Pi \cap \fm)}$.  Since $T$ is equicharacteristic, $R_0$ is a field. It follows that $R_0$ is an N-subring of $T$ and $R_0 \cap \fp = (0)$. Let $C = \{\fp\} \cup \Ass(T)$, and note that $C$ and $R_0$ satisfy the conditions of \cref{lemma:N-subring-transcendental}.
Let $\fq \in C$ and define $D_{(\fq)} \subset T$ to be a full set of coset representatives of the elements $u + \fq$ of $T/\fq$ that are algebraic over $R_0/(R_0 \cap \fq) \cong R_0$.  Define $D = \bigcup_{\fq \in C}D_{(\fq)}$ and note that $D$ is countable. By hypothesis, $y \not\in \fq$ for all $\fq \in C$. Now use \cref{lemma:countable-prime-avoidance} to find a unit $t$ of $T$ such that $yt \not\in \bigcup \{\fq + r \, | \,\fq \in C, \, r \in D\}.$ Define $R = R_0[yt]_{(R_0[yt] \cap \fm)}$ and note that $R_0 \subseteq R \subseteq T$ and $yt \in R$. By our choice of $t$, $yt + \fq \in T/\fq$ is transcendental over $R_0/(R_0 \cap \fq) \cong R_0$ for all $\fq \in C$. It follows by \cref{lemma:N-subring-transcendental} that $R$ is an N-subring of $T$, $R \cap \fp = (0)$, and $R$ is countably infinite.  
\end{proof}

We use the next theorem to ensure that the ring $A$ we construct satisfies the condition that $IT \cap A = I$ for every finitely generated ideal $I$ of $A$. The theorem is analogous to \cref{lemma:c-in-IS} from the previous section.

\begin{theorem}\label{theorem:N-subring-c-in-IS}
Let $(T,\fm)$ be a complete local ring and let $\fp$ be a nonmaximal prime ideal of $T$.  Let $(R,R \cap \fm)$ be an N-subring of $T$ with $R \cap \fp = (0)$.  Let $I$ be a finitely generated ideal of $R$ and let $c \in IT \cap R$.  Then there exists an N-subring $(S, S \cap \fm)$ of $T$ such that $R \subseteq S \subseteq T$, $|S| = |R|$, $S \cap \fp = (0)$, $c \in IS$, and prime elements of $R$ are prime in $S$.
\end{theorem}

\begin{proof}
The result follows from \cite[Theorem 6]{Loe97} and the note after the proof of Lemma 10 of \emph{op. cit.} which asserts that prime elements of $R$ remain prime in the ring $S$.
\end{proof}

\cref{lemma:N-subring-adjoin-coset} enables us to adjoin an element of a given coset $u + \fm^2$ to an N-subring that will result in another N-subring.  The result is analogous to \cref{lemma:adjoin-coset} from the previous section.

\begin{lemma}\label{lemma:N-subring-adjoin-coset}
Let $(T,\fm)$ be a complete local ring with $\depth T \geq 2$ and let $\fp$ be a nonmaximal prime ideal of $T$. Let $(R,R \cap \fm)$ be an infinite N-subring of $T$ with $R \cap \fp = (0)$ and let $u \in T$.  Then there exists an N-subring $(S, S \cap \fm)$ of $T$ such that $R \subseteq S \subseteq T$, $u + \fm^2$ is in the image of the map $S \longrightarrow T/\fm^2$, $|S| = |R|$, $S \cap \fp = (0)$, and prime elements of $R$ are prime in $S$.
\end{lemma}

\begin{proof}
Define $C = \{\fp\} \cup \Ass(T) \cup \{\fq \in \Spec(T) \, | \, \fq \in \Ass(T/rT) \mbox{ with } 0 \neq r \in R\}$. Note that by our hypotheses, $M \not\in C$. Let $\fq \in C$ and let $D_{(\fq)}$ be a full set of coset representatives of the elements $z + \fq \in T/\fq$ that make $z + u + \fq$ algebraic over $R$. Define $D = \bigcup_{\fq \in C}D_{(\fq)}$.  Since $\fm \not\subseteq \fq$ for all $\fq \in C$, we have $\fm^2 \not\subseteq \fq$ for all $\fq \in C$. Use \cref{lemma:countable} if $R$ is countable and \cref{lemma:uncountable} if not, both with $I = \fm^2$, to find $x \in \fm^2$ such that $x \not\in \bigcup \{\fq + r \, | \,\fq \in C, \, r \in D\}.$  Define $S = R[x + u]_{(R[x + u] \cap \fm)}$.  Note that $R \subseteq S \subseteq T$ and $u + \fm^2$ is in the image of the map $S \longrightarrow T/\fm^2$. Since $x + u$ is transcendental over $R$, prime elements in $R$ are prime in $S$. By \cref{lemma:N-subring-transcendental}, $S$ is an N-subring of $T$, $|S| = |R|$, and $S \cap \fp = (0)$.
\end{proof}

As in the previous section, we build an ascending chain of subrings of $T$ to construct the ring $A$. In this section, of course, we use N-subrings instead of $\fp$-subrings. The next lemma shows that, under certain circumstances, the union of an ascending chain of N-subrings is again an N-subring.  The result is the analogous version of \cref{lemma:unioning} from the previous section.

\begin{lemma}\label{lemma:N-subring-unioning}
Let $(T,\fm)$ be a complete local ring, let $\fp$ be a nonmaximal prime ideal of $T$, and let $(R_0, R_0 \cap \fm)$ be an infinite N-subring of $T$ with $R_0 \cap \fp = (0)$.  Let $\Omega$ be a well-ordered index set with least element $0$ such that either $\Omega$ is countable or, for every $\alpha \in \Omega$, we have $|\{\beta \in \Omega \, | \, \beta < \alpha \}| < |T/\fm|$. Suppose $\{(R_{\beta}, R_{\beta} \cap \fm)\}_{\beta \in \Omega}$ is an ascending collection of quasi-local subrings of $T$ such that if $\gamma(\alpha) < \alpha$ then $R_{\alpha}$ is an N-subring of $T$ with $R_{\gamma(\alpha)} \subseteq R_{\alpha}$, $|R_{\alpha}| = |R_{\gamma(\alpha)}|$, $R_{\alpha} \cap \fp = (0)$ and prime elements of $R_{\gamma(\alpha)}$ are prime in $R_{\alpha}$, while if $\gamma(\alpha) = \alpha$ then $R_{\alpha} = \bigcup_{\beta < \alpha}R_{\beta}$. 

Then $S = \bigcup_{\beta \in \Omega} R_{\beta}$ satisfies the conditions to be an N-subring of $T$ except for possibly the cardinality condition (\ref{item:N-subring-1}), but satisfies the inequality $|S| \leq \sup(|R_0|, |\Omega|).$ Moreover, $S \cap \fp = (0)$ and elements which are prime in some $R_{\alpha}$ are prime in $S$.
\end{lemma}

\begin{proof}
The result follows from \cite[Lemma 6]{Hei93} and our assumption that $R_{\alpha} \cap \fp = (0)$ whenever $\gamma(\alpha) < \alpha$.
\end{proof}

We are now ready to construct an N-subring $S$ of $T$ with many of our desired properties.

\begin{lemma}\label{lemma:intermediate-N-subring}
Let $(T,\fm)$ be a complete local ring with $\depth T \geq 2$ and let $\fp$ be a nonmaximal prime ideal of $T$. Let $(R,R \cap \fm)$ be an infinite N-subring of $T$ with $R \cap \fp = (0)$ and let $u \in T$.  Then there exists an N-subring $(S, S \cap \fm)$ of $T$ such that\begin{enumerate}
    \item $R \subseteq S \subseteq T$,
    \item $u + \fm^2$ is in the image of the map $S \longrightarrow T/\fm^2$,
    \item $|S| = |R|$,
    \item $S \cap \fp = (0)$,
    \item Prime elements of $R$ are prime in $S$, and
    \item For every finitely generated ideal $I$ of $S$, we have $IT \cap S = I$.
\end{enumerate}
\end{lemma}

\begin{proof}
First use \cref{lemma:N-subring-adjoin-coset} to find an N-subring $(R_0, R_0 \cap \fm)$ of $T$ such that $R \subseteq R_0 \subseteq T$, $u + \fm^2$ is in the image of the map $R_0 \longrightarrow T/\fm^2$, $|R_0| = |R|$,  $R_0 \cap \fp = (0)$, and prime elements of $R$ are prime in $R_0$.  Let $$\Omega = \{(I,c) \, | \, I \mbox{ is a finitely generated ideal of } R_0 \mbox{ and } c \in IT \cap R_0\}.$$
Note that $|\Omega| = |R_0|$.  Well-order $\Omega$ so that $0$ is its least element, and so that it does not have a maximal element.  We inductively define an ascending collection $\{R_{\alpha}\}_{\alpha \in \Omega}$ of N-subrings of $T$.  We have already defined $R_0$.  Let $\alpha \in \Omega$ and assume that $R_{\beta}$ has been defined for all $\beta < \alpha$.  If $\gamma(\alpha) < \alpha$ and $\gamma(\alpha) = (I,c)$ then define $R_{\alpha}$ to be the N-subring of $T$ obtained from \cref{theorem:N-subring-c-in-IS} so that $R_{\gamma(\alpha)} \subseteq R_{\alpha} \subseteq T$, $c \in IR_{\alpha}$, $|R_{\alpha}| = |R_{\gamma(\alpha)}|$, $R_{\alpha} \cap \fp = (0)$, and prime elements of $R_{\gamma(\alpha)}$ are prime in $R_{\alpha}$. If, on the other hand, $\gamma(\alpha) = \alpha$, then define $R_{\alpha} = \bigcup_{\beta < \alpha}R_{\beta}$. By \cref{lemma:N-subring-unioning}, $R_1 = \bigcup_{\alpha \in \Omega}R_{\alpha}$ is an N-subring of $T$, $|R_1| = |R_0|$, $R_1 \cap \fp = (0)$, and prime elements of $R_0$ are prime in $R_1$. Now let $I$ be a finitely generated ideal of $R_0$ and let $c \in IT \cap R_0$.  Then $(I,c) = \gamma(\alpha)$ for some $\alpha \in \Omega$ with $\gamma(\alpha) < \alpha$. By construction, $c \in IR_{\alpha} \subseteq IR_1$.  It follows that $IT \cap R_0 \subseteq IR_1$ for every finitely generated ideal $I$ of $R_0$.

Repeat the above construction replacing $R_0$ with $R_1$ to find an N-subring $R_2$ of $T$ such that $R_1 \subseteq R_2 \subseteq T$, $|R_2| = |R_1|$, $R_2 \cap \fp = (0)$, prime elements of $R_1$ are prime in $R_2$, and $IT \cap R_1 \subseteq IR_2$ for every finitely generated ideal $I\subseteq R_1$.  Continue to obtain a chain of N-subrings $R_1 \subseteq R_2 \subseteq R_3 \subseteq \cdots$ of $T$ such that, for all $n$, $|R_n| = |R|$, $R_n \cap \fp = (0)$, prime elements of $R_n$ are prime in $R_{n + 1}$, and, for every finitely generated ideal $I\subseteq R_n$, we have $IT \cap R_n \subseteq R_{n + 1}$. Let $S = \bigcup_{i = 1}^{\infty}R_i$. By \cref{lemma:N-subring-unioning}, $S$ is an N-subring of $T$, $|S| = |R|$, $S \cap \fp = (0)$, and prime elements of $R$ are prime in $S$. Now suppose that $I\subseteq S$ is a finitely generated ideal, say $I = (a_1, \ldots ,a_m)$ for some $a_i \in S$. Let $c \in IT \cap S$. Then there is an $r$ such that $c \in R_r$ and $a_i \in R_r$ for all $i \in \{1,2, \ldots ,m\}$. Now, $c \in (a_1,a_2, \ldots,a_m)T \cap R_r \subseteq (a_1,a_2, \ldots,a_m)R_{r+1} \subseteq I$. Hence $IT \cap S = I$, and it follows that $S$ is the desired N-subring of $T$.
\end{proof}

We now have the tools we need to prove one of the main results of this section.

\begin{theorem}\label{theorem:p-subring-ufd-precompletion}
Let $(T,\fm)$ be a complete local equicharacteristic ring with $\depth T \geq 2$ and let $\fp$ be a nonmaximal prime ideal of $T$.  Let $y$ be a regular element of $T$ such that $y \in \fm \setminus \fp$.  Then there exists a local UFD $(A, A \cap \fm)$ such that $\wh{A} \cong T$, $A \cap \fp = (0)$, and $yt \in A$ for some unit $t$ of $T$.
\end{theorem}

\begin{proof}
Let $\Omega = T/\fm^2$, and well-order $\Omega$ so that every element of $\Omega$ has fewer than $|\Omega$| predecessors.  Let $0$ denote the least element of $\Omega$. Let $(R_0, R_0 \cap \fm)$ be the countably infinite N-subring of $T$ given by \cref{lemma:N-subring-base-ring} so that $R_0 \cap \fp = (0)$ and $yt \in R_0$ for some unit $t \in T$. We now recursively define an ascending collection $\{R_{\alpha}\}_{\alpha \in \Omega}$ of infinite quasi-local subrings of $T$. Note that $R_0$ has already been defined. 
 Let $\alpha \in \Omega$ and assume that $R_{\beta}$ has been defined for every $\beta < \alpha$.  If $\gamma(\alpha) < \alpha$, define $R_{\alpha}$ to be the N-subring obtained from \cref{lemma:intermediate-N-subring} so that $R_{\gamma(\alpha)} \subseteq R_{\alpha} \subseteq T$, $\gamma(\alpha)$ is in the image of the map $R_{\alpha} \longrightarrow T/\fm^2$, $|R_{\alpha}| = |R_{\gamma(\alpha)}|$, $R_{\alpha} \cap \fp = (0)$, prime elements of $R_{\gamma(\alpha)}$ are prime in $R_{\alpha}$, and $IT \cap R_{\alpha} = I$ for every finitely generated ideal $I\subseteq R_{\alpha}$.  If $\gamma(\alpha) = \alpha$, define $R_{\alpha} = \bigcup_{\beta < \alpha}R_{\beta}$.  Now define $A = \bigcup_{\alpha \in \Omega}R_{\alpha}$.  By \cref{lemma:N-subring-unioning}, $A$ is a UFD and $A \cap \fp = (0)$. By construction, the map $A \longrightarrow T/\fm^2$ is onto. Let $I\subseteq A$ be a finitely generated ideal and let $c \in IT \cap A$. Write $I = (a_1, a_2, \ldots,a_m)$ for $a_i \in A$ and choose $\alpha$ so that $\gamma(\alpha) < \alpha$ and $c, a_1, a_2, \ldots ,a_m \in R_{\gamma(\alpha)}$. Now, $c \in (a_1,a_2, \ldots ,a_m)T \cap R_{\alpha} = (a_1, a_2, \ldots,a_m)R_{\alpha} \subseteq I$.  It follows that $IT \cap A = I$ for all finitely generated ideals $I$ of $A$. By \cref{proposition:heitmann-completion}, we have $\wh{A} \cong T$.
\end{proof}

The following should be compared with \cite[Theorem 23]{LR01}.
\begin{theorem}\label{theorem:f-regular-precompletion-ufd}
    Let $(T,\fm)$ be a complete local reduced Gorenstein ring with $\dim T\geq 2$ of prime characteristic $p>0$. Let $I\subseteq T$ be a parameter ideal and let $y\in T$ generate the socle of $T/I$. Suppose that $y^q\not\in I^{[q]}$ for all $q=p^e\gg 0$, and further assume that there exists a nonmaximal $\fp\in\Spec^\circ T$ such that $M_{I,y}\subseteq\fp$. Then there exists an $F$-regular local UFD $(A,\fm\cap A)$ such that $\wh{A}\cong T$.
\end{theorem}
\begin{proof}
    By prime avoidance as in the proof of \cref{theorem:f-regular-precompletion-non-gor} we may assume that $y$ is a regular element of $T$ avoiding $\fp$. We obtain the stated result \emph{mutatis mutandis} as in the proofs of \cite[Theorem 23]{LR01} and \cref{theorem:f-regular-precompletion-non-gor}, using \cref{theorem:p-subring-ufd-precompletion} in place of \cite[Theorem 22]{LR01} (resp. in place of \cref{theorem:p-subring-precompletion}).
\end{proof}
An identical proof strategy of \cref{corollary:f-regular-power-series-precompletion-non-gor} yields the following:
\begin{corollary}\label{corollary:f-regular-power-series-precompletion-ufd}
    Let $(T,\fm)$ be a complete local reduced Gorenstein ring of prime characteristic $p>0$ and $\dim T\geq 1$. Let $I\subseteq T$ be a parameter ideal and let $y\in T$ generate the socle of $T/I$. Suppose that $y^{p^e}\not\in I^{[p^e]}$ for all $e$ (for example, we may take $T$ to be $F$-pure). Then there exists an $F$-regular local UFD $A$ such that $\wh{A}\cong T\llbracket X\rrbracket$.
\end{corollary}

\subsection{A demonstration of \texorpdfstring{\cref{maintheorem:precompletion-2}}{Theorem D}}\label{section:fermat}
In this subsection we provide a particular example exhibiting the statements of \cref{theorem:f-regular-precompletion-ufd} and \cref{corollary:f-regular-power-series-precompletion-ufd} which employs the construction of \cite{HRW97}. This is an extension of the Fermat cubic case considered in \cite[Example 2]{LR01}. Our presentation here is more general and has a slightly simpler proof, while also adapting to give UFD examples in sufficiently high dimensions for parafactoriality reasons (see \cref{corollary:fermat-quintic}).
\begin{setting}\label{setting:fermat}
    Let $\sk$ be a perfect field of prime characteristic $p>0$, and let $$R=\sk[x,t,y_2,\ldots,y_d]_{(x,t,y_2\ldots, y_d)}$$ be the localized polynomial ring in $d+1$ indeterminates over $\sk$. For each $1\leq i\leq d$, fix the power series
    \begin{align*}
        \tau_i:=\sum\limits_{n=1}^\infty a_{i,n} x^n\in x \sk\llbracket x\rrbracket
    \end{align*}
    which we choose to be algebraically independent over $\sk (x)$. Define
    \begin{align*}
        \rho:=(t+\tau_1)^d - \sum\limits_{i=2}^d (y_i+\tau_i)^d \in \wh{R}=\sk \llbracket x,t,y_2,\ldots, y_d\rrbracket.
    \end{align*}
    Let $B:=\Frac(R)(\rho)\cap \wh{R}$ and let $A:=B/(\rho)$. Define the \emph{$n$th endpiece of $\rho$} to be 
    \begin{align*}
        \rho_n:=\frac{1}{x^n}\left(\rho-\left(\left(t+\sum\limits_{j=0}^n a_{1,j} x^j\right)^d - \sum\limits_{\ell=2}^d \left(y_\ell +\sum\limits_{j=1}^n a_{\ell,j}x^j\right)^d\right)\right)
    \end{align*}
    and observe for each $n$ that $\rho_n=x\rho_{n+1}+r_n$ for some $r_n\in R$.  Finally, define $B_n:=R[\rho_n]_{(x,t,y_2,\ldots, y_d,\rho_n)}$.
\end{setting}

\begin{proposition}
    Assume the notation of \cref{setting:fermat}. Then $B$ (and hence $A$) is a Noetherian local ring. Moreover, $B$ may be realized as the colimit $B=\lim\limits_{\longrightarrow_n} B_n$, and $\wh{B}=\wh{R}$.
\end{proposition}
\begin{proof}
The embedding
    \begin{align*}
        \varphi:R[\rho]\hookrightarrow R[\tau_1,\ldots, \tau_d]
    \end{align*}
is flat if and only if so too is the $R$-algebra extension
\begin{align}\label{eq:fermat-1}
    \varphi':R[s]\rightarrow R[w_1,\ldots, w_d]\cong\frac{R[s,w_1,\ldots, w_d]}{\left(s-((t+w_1)^d - \sum\limits_{i=2}^d (y_i+w_i)^d\right)}.
\end{align}
The coefficients of the hypersurface in (\ref{eq:fermat-1}) (viewed as a polynomial in $R[w_1,\ldots, w_d]$) generate the unit ideal. $\varphi'$ is then flat by \cite[Theorem 22.6]{Mat86}, hence so is $\varphi$. The conclusions of the Proposition now follow from \cite[Theorem 3.2]{HRW97}.
\end{proof}

\begin{proposition}
    Assume the notation of \cref{setting:fermat}. Then $\wh{A}$ is a Gorenstein local ring which is not $F$-regular.
\end{proposition}
\begin{proof}
    Consider the parameter ideal $J=(x,y_2+\tau_2,\ldots,y_d+\tau_d)\subseteq \wh{A}$. One checks that $$(t+\tau_1)^{d-1}\in J^*\setminus J,$$ hence $\wh{A}$ is not $F$-regular.
\end{proof}
In the following proof, $I^F=\{r\in A\mid r^{p^e}\in I^{[p^e]}\text{ for all } e\gg 0\}$ denotes the \emph{Frobenius closure} of an ideal $I\subseteq A$.

\begin{theorem}
Assume the notation of \cref{setting:fermat}. Then the following are equivalent:
        \begin{enumerate}
            \item $p\equiv 1\mod d$.\label{item:fermat-item-1}
            \item $\wh{A}$ is $F$-pure.\label{item:fermat-item-2}
            \item $A$ is $F$-regular.\label{item:fermat-item-3}
        \end{enumerate}
\end{theorem}
\begin{proof}
We first note the isomorphism 
\begin{align}
    \wh{A}=\frac{\sk\llbracket x, t,y_2,\ldots,y_d\rrbracket}{(\rho)}\cong\frac{\sk\llbracket t+\tau_1,y_2+\tau_2,\ldots,y_d+\tau_d\rrbracket}{(\rho)}\llbracket x\rrbracket\cong \frac{\sk\llbracket t,y_2,\ldots, y_d\rrbracket}{(t^d-y_2^d-\cdots-y_d^d)}\llbracket x\rrbracket.\label{eqn:fermat-isomorphism}
\end{align}
A well-known consequence of Fedder's criterion \cite{Fed83} says that $D:=\frac{\sk\llbracket t,y_2,\ldots, y_d\rrbracket}{(t^d-y_2^d-\cdots-y_d^d)}$ is $F$-pure if and only if $p\equiv 1 \mod d$. To see this, let $\fm$ denote the maximal ideal of $D$ and note that the coefficient of $(ty_2\cdots y_d)^{p-1}$ in the monomial expansion of $(t^d-y_2^d-\cdots-y_d^d)^{p-1}$ is nonzero if and only if $p\equiv 1\mod d$. In this case $(t^d-y_2^d-\cdots-y_d^d)^{p-1}\not\in\fm^{[p]}$. Furthermore, when $p\not\equiv 1\mod d$, one checks that $(t^d-y_2^d-\cdots-y_d^d)^{p-1}\in \fm^{[p]}$. Since the closed fiber of $D\to D\llbracket x\rrbracket\cong \wh{A}$ is regular, it follows that $\wh{A}$ if $F$-pure (for example, by \cite[Theorem 7.4]{MP21}). This proves (\ref{item:fermat-item-1}) $\Leftrightarrow$ (\ref{item:fermat-item-2}). On the other hand, (\ref{item:fermat-item-3}) $\Rightarrow$ (\ref{item:fermat-item-2}) follows immediately from \cite[Corollary 2.3]{MP21}.

We now show the implication (\ref{item:fermat-item-2}) $\Rightarrow$ (\ref{item:fermat-item-3}), so suppose that $\wh{A}$ is $F$-pure. Consider the parameter ideal $I=(x,y_2,\ldots, y_d)\subseteq A$. Since $A$ is Gorenstein, to show that $A$ is $F$-regular it suffices to show that $I=I^*$. The socle of the Gorenstein artinian ring $A/I$ is generated by $t^{d-1}$. To show that $I=I^*$, it suffices to show that $t^{d-1}\not\in I^*$. From the isomorphism in (\ref{eqn:fermat-isomorphism}) we may apply \cref{proposition:multiplier} to see that either 
\begin{align*}
    M_{\wh{I},t^{d-1}}\subseteq (t+\tau_1,y_2+\tau_2,\ldots, y_d+\tau_d)\wh{A}\text{ or }M_{\wh{I},t^{d-1}}=\wh{A}.
\end{align*} 
However, $$B\cap (t+\tau_1,y_2+\tau_2,\ldots, y_d+\tau_d)\wh{R}=(\rho)B,$$ so $A\cap (t+\tau_1,y_2+\tau_2,\ldots, y_d+\tau_d)\wh{A}=(0)$ and the containment $M_{I,t^{d-1}}\subseteq M_{\wh{I},t^{d-1}}$ forces $M_{I,t^{d-1}}=A$. Now, if $t^{d-1}\in I^*$, we must have $t^{(d-1)q}\in I^{[q]}$ for all $q=p^e\gg 0$. This implies that $t^{d-1}\in I^F$. Since $\wh{A}$ is $F$-pure, so too is $A$ by \cite[Corollary 2.3]{MP21}, so $I=I^F$. Since $t^{d-1}\not\in I$, we have a contradiction.
\end{proof}

For the following corollary we first recall Grothendieck's parafactoriality theorem as well as the special case that shall be used. A short algebraic proof may be found in \cite{CL94}.

\begin{theorem}\cite[XI Corollaire 3.10 and XI Th\'{e}or\`{e}me 3.13(ii)]{SGA2}\label{theorem:parafactorial}
    If $(R,\fm)$ is a complete intersection with $\dim(R)\geq 4$, then $R$ is parafactorial. In particular, any hypersurface with an isolated singularity and of dimension at least four is a UFD.
\end{theorem}

\begin{corollary}\label{corollary:fermat-quintic}
    Assume the notation of \cref{setting:fermat} with $d=5$ and $p\equiv 1\mod 5$. Then $A$ is an $F$-regular UFD such that $\wh{A}$ is not $F$-regular.
\end{corollary}
\begin{proof}
    The only claim left to prove is that $A$ is a UFD, and in fact we claim that $\wh{A}$ is a UFD. With our assumptions on the characteristic, the Jacobian criterion says that the singular locus of the ring $$\frac{\sk[t,y_2,y_3,y_4,y_5]}{(t^5-y_2^5-y_3^5-y_4^5-y_5^5)}$$ is defined by its homogeneous maximal ideal. Completing at the homogeneous maximal ideal tells us that $\frac{\sk\llbracket t,y_2,y_3,y_4,y_5\rrbracket}{(t^5-y_2^5-y_3^5-y_4^5-y_5^5)}$ is a four-dimensional isolated singularity, hence a UFD by \cref{theorem:parafactorial}. It follows from \cite[Satz 2]{Sch67} that $\wh{A}\cong\frac{\sk\llbracket t,y_2,y_3,y_4,y_5\rrbracket}{(t^5-y_2^5-y_3^5-y_4^5-y_5^5)}\llbracket x\rrbracket$ is a UFD. Finally, $A$ is a UFD by \cite[Ch. VII, \S3.6, Proposition 4]{Bou89} as claimed.
\end{proof}

\begin{remark}
    The example provided by \cref{corollary:fermat-quintic} is (analytically) a UFD. However, \cref{maintheorem:precompletion-2} guarantees the existence of an $F$-regular local UFD $A$ whose completion is neither $F$-regular nor a UFD, and we are not aware of a specific example demonstrating this.
\end{remark}

\section{Noncatenary prime characteristic splinters}\label{section:noncatenary} The goal of this section is to construct the eponymous rings of the article --- that is, we prove \cref{maintheorem:noncatenary} as well as the noncatenary portion of \cref{maintheorem:precompletion-1}. The mere existence of a noncatenary weakly $F$-regular local ring is a straightforward consequence of \cref{theorem:f-regular-precompletion-non-gor} applied to a particular complete $F$-pure local ring which is not equidimensional. Here, the critical fact informing our constructions is \cite[Theorem 31.6]{Mat86} which prescribes the necessary condition that (in the notation of \cref{theorem:f-regular-precompletion-non-gor}) $T$ must be non-equidimensional in order for $A$ to be noncatenary.

\begin{theorem}\label{theorem:noncatenary}
    There exists a weakly $F$-regular local domain which is not catenary.
\end{theorem}
\begin{proof}
    Let $\sk$ be any field of prime characteristic $p>0$ and consider the ring $R=\frac{\sk\llbracket x,y,z,w\rrbracket}{(xy,xz)}$ with maximal ideal $\fn=(x,y,z,w)$. Consider the chains
    \begin{align}
        (x)\subsetneq (x+w,w)\subsetneq (x+w,y+w,w)\subsetneq \fn\label{eq:noncatenary-1}\\
        (y,z)\subsetneq (y+w,z+w,w)\subsetneq \fn.\label{eq:noncatenary-2}
    \end{align}
Note that $R\cong T\llbracket w\rrbracket$ where $T=\frac{\sk\llbracket x,y,z\rrbracket}{(xy,xz)}$, and $\fn$ corresponds to $(\fm R, w)$ where $\fm=(x,y,z)\subseteq T$. $T$ is $F$-pure by Fedder's criterion \cite{Fed83}, so it is approximately Gorenstein. We may therefore apply \cref{corollary:f-regular-power-series-precompletion-non-gor} to the ring $T$ together with the regular elements $x+w,y+w,z+w,w\in R\setminus \fm R$. This yields a weakly $F$-regular domain $A$ such that $\wh{A}\cong R$ and such that $$r_1(x+w),r_2(y+w),r_3(z+w),r_4w\in A$$ for units $r_i\in R$. Intersecting the chains in (\ref{eq:noncatenary-1}) and (\ref{eq:noncatenary-2}) with $A$ yields the saturated chains
{\small
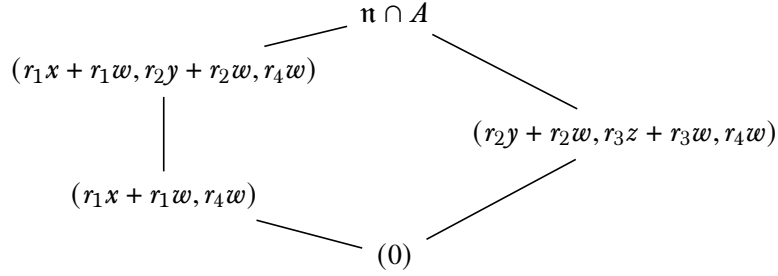
\begin{equation*}
    \begin{tikzcd}[row sep=tiny, column sep=tiny]
        &\fn\cap A & \\
        (r_1x+r_1w,r_2y+r_2w,r_4w)\arrow[dash]{ur}&&\\
        &&(r_2y+r_2w,r_3z+r_3w,r_4w)\arrow[dash]{uul}\\
        (r_1x+r_1w,r_4w)\arrow[dash]{uu}&&\\
        & (0)\arrow[dash]{uur}\arrow[dash]{ul} &
    \end{tikzcd}
\end{equation*}
}
from which we conclude that $A$ is noncatenary.
\end{proof}
Motivated by the proof of \cref{theorem:noncatenary}, we now turn to the remainder of \cref{maintheorem:precompletion-1} which manufactures our noncatenary ring $A$ systematically for a broader class of complete local rings $T$. To accomplish this, we require an adaptation of the results from \cref{section:non-gor} that serves to capture the prime avoidance idea presented above. Our argument hinges on ideas originating in \cite{CL04}. We first recall the following definition from \emph{op. cit.}

\begin{definition}\cite{CL04}
Let $(T,\fm)$ be a complete local ring and let $C\subseteq\Spec T$. Suppose that $(R,R \cap \fm)$ is a quasi-local subring of $T$ such that $|R| < |T|$ and $R \cap \fp = (0)$ for every $\fp \in C$. Then we call $R$ a \emph{small $C$-avoiding subring of $T$}, abbreviated henceforth as an \emph{SCA-subring of $T$.}
\end{definition}

\begin{remark}\label{remark:SCA-subring}
Suppose $T$, $\fp$, and $y_1, y_2, \ldots$ satisfy the conditions of \cref{lemma:base-ring}. Then $\dim T \geq 1$ so $T$ is uncountable, e.g. by \cite[Lemma 2.3]{CL04}. If $C = \Ass(T) \cup \{\fp\}$, then the countable subring obtained from \cref{lemma:base-ring} is an SCA-subring of $T$.
\end{remark}

\begin{theorem}\label{theorem:p-SCA-subring-precompletion}
    Let $(T,\fm)$ be a complete local ring and let $\Pi$ be the prime subring of $T$. Suppose that all elements of $\Pi$ are regular elements of $T$ and let $\fp$ be a nonmaximal prime ideal of $T$ such that $\Pi \cap \fp = (0)$.  Let $G' = \Ass(T) \cup \{\fp\}$ and let $G = \{P \in \Spec(T) \, | \, P \subseteq Q \mbox{ for some }Q \in G'\}.$  Let $y_1, y_2, \ldots$ be a countable collection of regular elements of $T$ with $y_i \in \fm \setminus \fp$ for all $i = 1,2, \ldots$. Then there exists a local domain $A$ such that
\begin{enumerate}
    \item $\wh{A} \cong T$;\label{theorem:p-SCA-subring-precompletion-1}
    \item $A \cap \fp = (0)$;\label{theorem:p-SCA-subring-precompletion-2}
    \item for every $i = 1,2, \ldots$, there is a unit $t_i$ of $T$ satisfying $y_it_i \in A$;\label{theorem:p-SCA-subring-precompletion-3}
     \item $\{\fq\in\Spec T\mid \fq\cap A=(0)\}=G$;\label{theorem:p-SCA-subring-precompletion-4}
     \item if $J$ is a nonzero prime ideal of $A$ then $T \otimes_A \kappa(J) \cong \kappa(J)$ where $\kappa(J) = A_J/JA_J$.\label{theorem:p-SCA-subring-precompletion-5}
\end{enumerate}
\end{theorem}

\begin{proof}
Define $C$ to be the maximal elements of $G$ with respect to inclusion. Let $(S, S \cap \fm)$ be the countably infinite $\fp$-subring of $T$ obtained from \cref{lemma:base-ring} so that, for every $i = 1,2, \ldots,$ there is a unit $t_i$ of $T$ satisfying $y_it_i \in S$. By \cref{remark:SCA-subring}, $S$ is an SCA-subring of $T$. The diligent reader may replace $R_0$ in the proof of \cite[Lemma 2.8]{CL04} with $S$ and check that the domain $A$ produced in \emph{op. cit.} satisfies the listed properties.
\end{proof}

The following two lemmas enable us to construct saturated chains in a complete non-equidimensional local ring with properties allowing for a noncatenary precompletion.

\begin{lemma}\label{lemma:chains}\cite[Lemma 2.8]{ABKLS19}
Let $(T,\fm)$ be a local ring with $\fm \not\in \Ass(T)$ and let $\fP$ be a minimal prime ideal of $T$ with dim$(T/\fP) = n$. Then there exists a saturated chain of prime ideals of $T$, $\fP \subsetneq \fq_1 \subsetneq \cdots \subsetneq \fq_{n - 1} \subsetneq \fm$ such that, for each $i = 1,2, \ldots ,n - 1$, $\fq_i \not\in \Ass(T)$ and $\fP$ is the only minimal prime ideal contained in $\fq_i$.
\end{lemma}

\begin{lemma}\label{lemma:specialchains}
Let $(T,\fm)$ be a catenary reduced local ring and let $\fP$ be a minimal prime ideal of $T$ with dim$(T/\fP) = n \geq 2$. Suppose $\fp$ is a nonmaximal prime ideal of $T$. Then there exists a saturated chain of prime ideals of $T$, $\fP\subsetneq \fq_1 \subsetneq \cdots \subsetneq \fq_{n - 1} \subsetneq \fm$ such that, for each $i = 1,2, \ldots ,n - 1$, $\fq_i \not\subseteq \fp$ and $\fP$ is the only minimal prime ideal contained in $\fq_i$.
\end{lemma}

\begin{proof}
If $\fP \not\subseteq \fp$ then \cref{lemma:chains} gives the desired chain, so suppose $\fP \subseteq \fp$. Let $$Y_1 = \{\fq \in \Spec(T) \mid \fP \subsetneq \fq \mbox{ is saturated and } \fq \not\subseteq \fp \}.$$ Suppose $Y_1$ is finite and let $Y_1 = \{\fq'_1, \ldots ,\fq'_s\}$. By prime avoidance, there is an $x \in \fm$ such that $x \not\in \fp$ and $x \not\in \fq'_j$ for $j = 1,2, \ldots ,s$. Let $\fq'/\fP$ be a minimal prime ideal of $(x + \fP)\in T/\fP$. Then $\fP \subsetneq \fq'$ is saturated and $x \in \fq'$. It follows that $\fq' \in Y_1$ and $\fq' \not\in \{\fq'_1, \ldots ,\fq'_s\}$, a contradiction.  Therefore, $Y_1$ is infinite. Now suppose that $\fq$ is a prime ideal of $T$ such that $\fP \subsetneq \fq$ is saturated and $\fq\supseteq\fP'$ where $\fP'$ is a minimal prime ideal of $T$ such that $\fP' \neq \fP$. Then $\fq$ is a minimal prime ideal of $\fP + \fP'$, of which there are only finitely many. It follows that there is a prime ideal $\fq_1$ of $T$ such that $\fq_1 \in Y_1$ and $\fP$ is the only minimal prime ideal of $T$ contained in $\fq_1$. 

Let $n>2$ and suppose that the saturated chain $\fP\subsetneq \fq_1\subsetneq\cdots\subsetneq \fq_{\ell}$ as in the statement of the lemma has already been defined up to $\ell<n-1$. Let $$Y_\ell = \{\fq \in \Spec(T) \mid \fq_1 \subsetneq \fq_2\subsetneq\cdots\subsetneq \fq_\ell\subsetneq\fq \mbox{ is saturated} \}$$ and note that $\fq\nsubseteq\fp$ for every $\fq\in Y_\ell$. Since $T$ is catenary, $\fq_\ell \neq \fm$ and $\fq_\ell \subseteq \fm$ is not saturated. Thus there are prime ideals $\fq'$ and $\fq''$ of $T$ such that $\fq_\ell \subsetneq \fq' \subsetneq \fq''$ is saturated. Since $T$ is Noetherian, there are infinitely many prime ideals strictly between $\fq_\ell$ and $\fq''$, and since $T$ is catenary, if $\fq$ is a prime ideal of $T$ strictly between $\fq_\ell$ and $\fq''$, then $\fq_\ell \subsetneq \fq$ is saturated. It follows that $Y_\ell$ is infinite. Now suppose that $\fq$ is a prime ideal of $T$ such that $\fq \in Y_\ell$ and $\fq$ contains $\fP'$ where $\fP'$ is a minimal prime of $T$ such that $\fP' \neq \fP$. Then $\fq$ is a minimal prime of $\fq_\ell + \fP'$ of which there are only finitely many.  Thus there is a prime ideal $\fq_{\ell+1}$ of $T$ such that $\fq_1 \subsetneq \fq_2\subsetneq\cdots\subsetneq\fq_{\ell+1}$ is saturated and $\fP$ is the only minimal prime of $T$ contained in $\fq_1,\ldots,\fq_{\ell+1}$. Since $T$ is catenary, this procedure stops when $\ell=n-1$.
\end{proof}

\begin{theorem}\label{theorem:f-regular-precompletion-non-gor-noncatenary}
    Assume the hypotheses and notation of \cref{theorem:f-regular-precompletion-non-gor}. Further assume that $T$ is not equidimensional and that $\depth T\geq 2$. Then there exists a weakly $F$-regular local domain $(A,\fm\cap A)$ which is not catenary such that $\wh{A}\cong T$.
\end{theorem}
\begin{proof}
    By the same prime avoidance argument as in \cref{theorem:f-regular-precompletion-non-gor}, we may assume that the socle generators are regular elements avoiding $\fp$. We are then able to apply \cref{theorem:p-SCA-subring-precompletion} to this situation, with $G=\min(T)\cup\{\fp\}$. The local domain $A$ guaranteed by this procedure is weakly $F$-regular by the same proof as \cref{theorem:f-regular-precompletion-non-gor}. To show that $A$ is noncatenary, we will produce a prime $\fQ\in\Spec A$ such that $\Height\fQ+\dim(A/\fQ)<\dim A$. We follow the argument of \cite[Theorem 2.10]{ABKLS19}, sketching the main ideas for the reader.

    Let $\fp_0\in\min(T)$ be a minimal prime of non-maximal dimension, say $$2\leq n:=\dim(T/\fp_0)<\dim T.$$ By \cref{lemma:specialchains}, there exists a saturated chain 
\begin{equation}
    \fp_0\subsetneq \fq_1\subsetneq\fq_2\subsetneq\cdots\subsetneq \fq_{n-1}\subsetneq\fm\label{equation:f-regular-precompletion-non-gor-noncatenary-1}
\end{equation}
    where $\fp_0$ is the only minimal prime contained in each $\fq_i$. In particular, $\Height\fq_{n-1}=n-1$ and $\dim (T/\fq_{n-1})=1$. Moreover, since $\fq_i\nsubseteq\fp$, we see that $\fq_{n-1}\not\in G$. It follows that $\fq_{n-1}\cap A\neq (0)$ by condition (\ref{theorem:p-SCA-subring-precompletion-4}) of \cref{theorem:p-SCA-subring-precompletion}.

    Suppose now that $\fQ:=\fq_{n-1}\cap A\subsetneq \fP\in\Spec A$. By \cite[Remark 2.4]{ABKLS19}, there is a one-to-one inclusion-preserving correspondence between nonzero primes of $A$ and the primes of $T$ avoiding $G$, hence $\fq_{n-1}\subsetneq \fP T\in\Spec T$. Since the chain (\ref{equation:f-regular-precompletion-non-gor-noncatenary-1}) is saturated, $\fP T=\fm$, so $\fP=\fm\cap A$. It follows that $\dim(A/\fQ)=1$, hence
    \begin{equation}
        \Height\fQ+\dim(A/\fQ)\leq \Height\fq_{n-1}+\dim(T/\fq_{n-1})=n<\dim A,
    \end{equation}
    as desired.
\end{proof}

\begin{remark}\label{remark:cardinality}
 We mention one subtle difference between \cref{theorem:f-regular-precompletion-non-gor,theorem:f-regular-precompletion-non-gor-noncatenary}. The ring $A$ which is produced by \cref{theorem:p-SCA-subring-precompletion} is always uncountable, whereas the analogous ring produced by \cref{theorem:p-subring-precompletion} is countable whenever $T/\fm$ is countable.
\end{remark}

\section{Questions}\label{section:questions}
We conclude the article by mentioning some cases which are not handled by our present techniques.
\begin{question}
    Can one give an explicit description of a noncatenary weakly $F$-regular local ring? In particular, is the example from \cite{Ogo80,Hei82} weakly $F$-regular if the underlying field has positive characteristic?
\end{question}

\begin{question}
    Does there exist a noncatenary weakly $F$-regular ring which is also a UFD? For example, can the techniques of \cref{section:gor} be extended to the non-Gorenstein case to allow for the methods of \cref{section:non-gor,section:noncatenary} to be used?
\end{question}

\begin{question}
    Condition (\ref{theorem:theoremA-1}) is necessary in \cref{maintheorem:precompletion-1} since $F$-purity is a necessary condition. Is condition (\ref{theorem:theoremA-2}) also necessary?
\end{question}

\begin{question}
    Do results analogous to Theorems \ref{maintheorem:noncatenary}--\ref{maintheorem:precompletion-2} hold for splinters in mixed characteristic?
\end{question}

\printbibliography
\end{document}